\documentclass[12pt]{amsart}
\usepackage[top=30truemm,bottom=30truemm,left=25truemm,right=25truemm]{geometry}
\usepackage{txfonts}
\usepackage{mathrsfs}

\usepackage{caption}
\usepackage{graphicx}
\usepackage{float}
\usepackage{subfig}
\usepackage{overpic}

\usepackage{bm}
\usepackage{amsfonts,amssymb}
\usepackage{dsfont}
\usepackage{extarrows}
\usepackage{amsmath}
\usepackage{mathrsfs}
\usepackage{enumerate}
\usepackage{amscd}
\usepackage[all]{xy}
\usepackage{hyperref}

\theoremstyle{plain} 

\theoremstyle{definition} 

\newtheorem{thm}{Theorem}[section]

\newtheorem{lem}[thm]{Lemma}

\theoremstyle{definition}

\theoremstyle{remark}

\newcommand{\be}{\begin{equation}}
\newcommand{\ee}{\end{equation}}
\newcommand{\bea}{\begin{eqnarray}}
\newcommand{\eea}{\end{eqnarray}}
\newcommand{\ben}{\begin{eqnarray*}}
	\newcommand{\een}{\end{eqnarray*}}
\newcommand{\bt}{\begin{split}}
	\newcommand{\et}{\end{split}}
\newcommand{\bet}{\begin{equation}}

\newcommand{\mr}{\mathbb{R}}
\newcommand{\ra}{\rightarrow}
\newcommand{\beq}{\begin{equation*}}
\newcommand{\eeq}{\end{equation*}}
\newcommand{\bn}{\begin{enumerate}}
\newcommand{\en}{\end{enumerate}}
\begin{document}

\title[]
{Numerical calculation method for function integration on submanifolds of $\mr^n$ or compact Riemannian manifolds}

\author[F. Deng]{Fusheng Deng}
\address{Fusheng Deng: \ School of Mathematical Sciences, University of Chinese Academy of Sciences\\ Beijing 100049, P. R. China}
\email{fshdeng@ucas.ac.cn}

\author[G. Huang]{Gang Huang}
\address{Gang Huang: \ School of Mathematical Sciences, University of Chinese Academy of Sciences\\ Beijing 100049, P. R. China}
\email{huanggang21@mails.ucas.ac.cn}

\author[Y. Wu]{Yingyi Wu}
\address{Yingyi Wu: \ School of Mathematical Sciences, University of Chinese Academy of Sciences\\ Beijing 100049, P. R. China}
\email{wuyy@gucas.ac.cn}


\begin{abstract}
    In this paper, we present a method for digitally representing the "volume element" 
    and calculating the integral of a function on compact hypersurfaces with or without boundary, 
    and low-dimensional submanifolds in $\mathbb{R}^n$. 
    We also extend such calculation to hypersurfaces in compact Riemannnian manifolds.
\end{abstract}

\maketitle

\section{Introduction}\label{sec:intro}
In calculus, it is very important to compute the integral of a function over a manifold.
When considering the integral of a function over a smooth manifold, 
it becomes evident that differential forms possess the necessary properties for intrinsically defining integrals (\cite{Lee00}). 
Specifically, we seek a k-form on a smooth manifold as a kind of "volume element" that can be integrated in a coordinate-independent way over submanifolds of dimension k. 
On the other hand, on a Riemannian manifold, the Riemannian metric can induce a measure and hence one can define the integral 
of a function on the manifold.
A classical routine for the definition of the integral of a function $f$ on a compact Riemannian manifold $(M,g)$ can be given by the following steps:
\bn[(1)]
\item cut $M$ into small pieces say $K_1,\cdots, K_N$,
\item choose points $p_j\in K_j\ (1\leq j\leq N)$ in an arbitrary way,
\item consider the finite sum $$\sum^N_j f(p_j)\mu(K_j),$$ where $\mu$ is the measure on $M$ induced from the Riemannian metric $g$,
\item then the integral is defined to be a limit
$$\int_Mf=\lim_{\max_{j}\{\text{diam}(K_j)\}\ra 0}\sum^N_j f(p_j)\mu(K_j),$$
where $\text{diam}(K_j)$ is the diameter of $K_j$ with respect to $g$.
\en

Unfortunately, in practical applications, it is almost impossible to calculate the integral by the above process.
The key trouble is that in general we do not know how to cut $M$ into pieces in a workable way and do not know how to calculate $\mu(K_j)$, 
the measure of the pieces involved in the above finite sum.

The main purpose of the present work is to propose some digital method to overcome the above difficulty for submanifolds of $\mr^n$ or even 
general Riemannian manifolds. 
We interpret the main results via a special case of Riemannian manifolds.
For a compact hypersurface $M$ in $\mr^n$ and a sample of points $\{p_j\}^N_{j=1}$ in $M$, 
we try to find a way to endow each $p_j$ ``a volume element'' $\mu_j$ and approximate the integral $\int_M f$ by the finite sum
$\sum_{j=1}^Nf(p_j)\mu_j$.
From the construction, one can expect that $\sum_{j=1}^Nf(p_j)\mu_j$ is very close to $\int_Mf$ if the sample $\{p_j\}^N_{j=1}$
is sufficiently dense in $M$.

Our method is motivated by the work in \cite{LXSW22}.
The authors of that paper aim to reconstruct a closed surface in $\mr^3$ via the Gauss formula from 
a given point-cloud, without knowing normal of the surface at these points.
Utilizing the Gauss formula, they define an equation to obtain consistently oriented outward-pointing surface elements for surface reconstruction, which is called Parametric Gauss Reconstructions(PGR). 

We begin the present paper with a discussion of the above work about how to obtain volume elements, 
but our objective has shifted from reconstructing surfaces from unoriented point clouds to numerically representing the volume element
associated to a given point-cloud. 
We will use the divergence theorem on Riemannian manifolds (\cite{Lee00}) as a tool, with Gauss's formula as a special example. 
The divergence theorem provides an integral formula for the indicator function, which is related to the fundamental solution of Laplace operator
and is called the double layer potential in potential theory (\cite{F95}).  
By viewing the volume elements as unkown parameter, we define the discrete version of the indicator function 
via a linear system constructed from the divergence theorem. 
As a result, the volme elements can be directly estimated from the linear system, 
which enables us to calculate the integral of functions on the involved surface, as explained above.

 The method is theoretically feasible. One the other hand, its effectiveness and
 stability are also important topics. 
 But these are not the main focus of this paper, so they will not be discussed in depth here. 

We first develop a numerical method for computing the intrgral of a function on a closed hypersurface in $\mathbb{R}^n$,
based on the ideal in \cite{LXSW22} as mentioned above. 
On the other hand, in reality, the majority of surfaces we commonly encounter are surfaces with boundaries rather than closed surfaces. 
When considering  surfaces with boundaries, our main insight is to create a duplicate of the surface through thickening the original surface. 
Then the original integral can be calculated via the newly constructed surface without boundary. 
We also consider submanifolds in $\mr^n$ of higher codimension.
For this case, our trick is to transfer the consideration from a submanifold to the boundary of its tubular neighborhoods. 

In principle, the above consideration for submanifolds in $\mr^n$ can be generalized to submanifolds in general Riemannian manifolds.
This paper provides an explicit integral formula for the indicator function of a relatively compact domain with smooth boundary in a Riemannian manifold. This method is particularly effective if we know the explicit formula of the fundamental solution of the Laplace operator on the Riemannian manifold. 
In conclusion, we present an algorithm for calculating integrals of functions on the boundaries related compact domains in general Riemannian manifolds.

Finally we point out that it is possible to calculate the integral of functions on a Riemannian manifold,
without embedding it into some other ambient spaces.
This will involves other integral formulas for functions on Riemannian manifolds 
that are also related the Green functions or heat kernel on Riemannian manifolds.
This topic will be investigated further in forthcoming works.

%

The paper is organized as follows. 
Section 2 provides an overview of some related concepts and facts. In Section 3, we consider integrals of functions on submanifolds in $\mathbb{R}^n$, and in the final Section 4 we consider integrals of functions on the boundaries related compact domains in general Riemannian manifolds.

\textbf{Acknowledgements.}
The authors are very grateful to Dr. Yifei Feng for drawing our attention to the work in \cite{LXSW22} by giving a wonderful report on this paper in our joint seminar,
and to Dr. Hongyu Ma and Professor Liyong shen for helpful discussions.
This research is supported by National Key R\&D Program of China (No. 2021YFA1003100) and the Fundamental Research Funds for the Central Universities.

\section{Preliminaries}\label{sec:pre}

This section briefly introduces some related basic concepts and results. 
Detailed treatments can be found in \cite{Lee00}, \cite{P16}.


\subsection{The divergence theorem}

Let $M$ be a smooth manifold with or without boundary. Recall that a Riemannian metric on $M$ is a smooth symmetric covariant 2-tensor field on $M$ that is positive definite at each point. A Riemannian manifold is a pair $(M, g)$, where $M$ is a smooth manifold and $g$ is a Riemannian metric on $M$.
If $g$ is a Riemannian metric on $M$, then for each $p \in M$, the 2-tensor $g_p$ is an inner product on $T_pM$.
In any smooth local coordinates ($x^i$), a Riemannian metric can be written as $$g = g_{ij}dx^{i}dx^{j},$$ where ($g_{ij}$) is a symmetric positive difinite matrix of smooth functions. The simplest example of a Riemannian metric is the Euclidean metric $\overline{g}$ on $\mathbb{R}^{n}$ which is given by $$\overline{g} = \delta_{ij}dx^{i}dx^{j},$$ 
where $\delta_{ij}$ is the Kronecker delta.

\begin{thm}\label{thm: the divergence theorem}(\cite{Lee00} Theorem 16.32)
    (The Divergence Theorem) Let $(M, g)$ be an oriented Riemannian manifold with boundary. 
    For any compactly supported smooth vector field $X$ on $M$,
    \begin{equation}
        \int_{M} divX dV_{M} = \int_{\partial M} g(X, \vec{n})dV_{\partial M}
    \end{equation}
    where $\vec{n}$ is the outward-pointing unit normal vector field along $\partial M$,  
    $dV_{M}$ is the Riemannian volume form of $M$ and $dV_{\partial M}$ is the Riemannian volume form of $\partial M$ induced from $g$.
\end{thm}

\subsection{Riemannian coordinates}
For a tangent vector $v \in T_{p}M$, there exists a unique geodesic $r_{v}(t)$ on $[0, 1]$ such that $r_{v}(0) = p, r_{v}^{'}(0) = v$. 
For $0<\delta<<1$, the exponential map 
$$v\mapsto \exp_p(v):=r_v(1)$$
is well defined on the open set
$$B_{p}(\delta ) = \{ v \in T_{p}M; |v|<\delta\}$$
in $T_pM$. If $(M,g)$ is complete and boundary free, the exponential map can be defined on the whole tangent space $T_pM$.

\begin{thm}\label{thm: exponential map diffeomorphism}(\cite{P16} Corollary 5.5.2)
    Let $(M, g)$ be a Riemannian manifold and $K\subset M$ be compact. There exists $\delta >0$ for every $p\in K$,the map $$\exp_p : B_{p}(\delta )\rightarrow U \subset M$$ is defined and a diffeomorphism onto its image.
    
\end{thm}

Define the function $r(x) = |\exp_{p}^{-1}(x)|$. we have

\begin{thm}\label{thm: Gauss lemma}(\cite{P16} Theorem 5.5.5)
    (Gauss lemma) On $(U, g)$ the function $r$ has gradient $\bigtriangledown r = \partial_{r}$, where $\partial_{r} = D\exp_{p}(\partial_{r})$.
\end{thm}

The open set 
$$B(p, \delta) = \exp_{p}(B_{p}(\delta)) = \{q\in M ; d(p,q)<\delta\}$$
in $M$ is called a geodesic ball of radius $\delta$, 
where $d(p,q)$ is the distance on $M$ induced by the Riemannian metric $g$.
By theorem \ref{thm: exponential map diffeomorphism}, when $\delta$ is sufficiently small, 
the exponential map $exp_{p}$ is diffeomorphism from $B_{p}(\delta)$ to $B(p, \delta) $. What's more, $\partial B(p, \delta) $ and $S^{n-1}$ is homeomorphism. 
We identity $T_{p}M$ with $\mathbb{R}^n$ by choosing an orthonormal basis of $T_pM$, 
then the diffeomorphism $\varphi = \exp_{p}^{-1} : U \rightarrow V$ induces a coordinate on $U$,
which is called the exponential or normal coordinates at $p$. 
Under a normal coordinate $(x^1,\cdots, x^n)$, the Reimannian metric can be locally given as 
\begin{equation}\label{eqn: exponential coordinates}
    g_{ij}=\delta_{ij}+o(|x|^2),
\end{equation}
which means that the Riemannian metric can be always approximated by the flat metric locally up to the first order (see \cite{Lee00} Lemma 5.5.7).

\subsection{Two more lemmas}

\begin{lem}\cite{F95}\label{lem: estimate of normal vector}
    Let $S$ is a closed hypersurface on $\mathbb{R}^n$. There is a constant $c > 0$ such that for all $x, y \in S$, $$|(x-y)\cdot \vec{n}(y)| \leq c|x-y|^2$$ where $\vec{n}(y)$ is the outward unit normal vector at $y$ on $S$.
\end{lem}

The following lemma will be used in the estimate of the fundamental solution to the Laplace equation in our discussion.

\begin{lem}\label{lem: the estimate of integration} (\cite{A82} Proposition 4.12)
    Let $\Omega$ be a bounded open set of $\mathbb{R}^n$ and Let $X(p,q)$ and $Y(p,q)$ be continuous functions defined on $\Omega \times \Omega$ minus the diagonal which satisfying
    $$|X(p,q)| \leq Const \times |d(p,q)|^{\alpha -n}$$ and  $$|Y(p,q)| \leq Const \times |d(p,q)|^{\beta -n}$$
    for some real numbers $\alpha , \beta$ in $(0, n)$. Then $$ Z(p,q) = \int_{\Omega} X(p,t)Y(t,q)dt_1\cdots dt_n$$ is continuous for $p \neq q$ and satisfies:
    $$|Z(p,q)| \leq Const \times |d(p,q)|^{\alpha  +\beta -n}, \quad if \quad \alpha  + \beta < n,$$
    $$|Z(p,q)| \leq Const \times [1 + |\ln d(p,q)|],\quad if \quad \alpha  + \beta = n,$$
    $$|Z(p,q)| \leq Const ,\quad if \quad \alpha  + \beta > n.$$
\end{lem}

\section{Numerical calculation method for function integration in submanifold in $ \mathbb{R}^n $}
The main interest in this section is to consider the following problem:
given a compact submanifold $M \subset \mathbb{R}^n$ and $f \in C^0(\mathbb{R}^n)$, can one compute $$\int_{M}f(x)dV?$$
where $dV$ is the volume form on $M$ induced from the standard flat metric on $\mr^n$.
In most practical contexts, we just know samples of points in $M$, but do not know the exact formula of $M$. 
If $M$ is a hypersurface without boundary, we can propose a appropriate digital method to calculate the above integration
from the information of two samples of points in $M$.
In the cases that $M$ has boundary or  higher codimension, 
we need additional information to calculate the integral, namely, the normals of $M$ at the sample points.

\subsection{Hypersurfaces in $\mathbb{R}^n$}\label{subsec:hypersurface without boundary}

Let $\Omega$ be a bounded connected open set in $\mathbb{R}^n (n>2)$ whose boundary $\partial \Omega $ is a smooth hypersurface. 
Then $\partial \Omega $ divides $ \mathbb{R}^n $ into two regions: an inner region $\Omega $ and an outer one, the interior of $\mathbb{R}^n\backslash \Omega $. 
Let $N$ be the outward-pointing unit normal vector field on $\partial\Omega$.

We now introduce and modify the method in \cite{LXSW22} to calculate the surface elements on $\partial\Omega$.
The work in \cite{LXSW22} only consider surfaces in $\mr^3$, but the idea can be generalized to higher dimensional cases.

The starting point is the following lemma,  which expresses the 0-1 indicator function of $\Omega$ in terms of certain integral on the boundary.

\begin{lem}[\cite{F95}]\label{lem:Gauss formula}
    Let $\Omega \subset \mathbb{R}^n$ be a bounded open set with smooth boundary $\partial\Omega$. Then 
\begin{equation}\label{eqn: indicator function in $R^n$}
  \chi (x):= -\int_{\partial \Omega} \nabla G(x-y)\cdot N(y)d\tau(y)=
    \begin{cases}
        0 ,\quad x=R^n
        \backslash \overline{\Omega}\\ \frac{1}{2} ,\quad x=\partial\Omega\\ 1 ,\quad x\in \Omega\\
    \end{cases},
\end{equation}
    where $d\tau(y)$ is the hypersurface area form of $\partial\Omega$, and $G(x-y)$ is the fundamental solution to the n-dimensional Laplace equation, i.e.
    \begin{equation}
        G(x-y)= \frac{|x-y|^{2-n}}{(n-2)\omega_n},
    \end{equation}
    with $\omega_n$ being the volume of the unit sphere in $R^n$.
\end{lem}

In general, we just know samples of points in $\partial \Omega$, but do not know the exact equation defining $\partial\Omega$.
Indeed the original purpose in \cite{LXSW22} is to find in an appropriate way the defining equation of $\partial\Omega$.
The process is as follows.  
Suppose $Y={\{ y_j\}}_{j=1}^{N_{Y}}\subset \partial\Omega$ is a point set that samples the hypersurface $\partial\Omega$, we get 
\begin{equation}\label{eqn: discreted indicator function in $R^n$}
    \chi(x)= -\int_{\partial \Omega}\nabla G(x-y)\cdot N(y)d\tau(y)\approx \sum_{j=1}^{N_{Y}}\frac{y_j-x}{\omega_n|x-y_j|^{n}}\cdot N(y_j)\tau(y_j).
\end{equation}
Then the remaining task is to solve out the volume elements $\tau(y_j)$.
If this is done, then we can view 
$$\sum_{j=1}^{N_{Y}}\frac{y_j-x}{\omega_n|x-y_j|^{n}}\cdot N(y_j)\tau(y_j)=1/2$$
as the defining equation of $\partial\Omega$.
In principle, this equation can be exact enough if the sample $\{y_j\}$ in $\partial\Omega$ is chosen good enough.

In the present work, the focus shifts from finding the defining equation of $\partial\Omega$ to calculating the integral $\int_{\partial\Omega}fd\tau$,
whose digital value can be approximated by $\sum_j f(y_j)\tau(y_j)$ once we can solve out the volume elements $\tau(y_j)$.

The way to calculate the volume elements $\tau(y_j)$ is as follows.
For each fixed $y_j$, 
$$\nabla G(x-y_j)=\frac{x-y_j}{\omega_n|x-y_j|^{n}}$$ 
is an $n$-dimensional vector-valued function of $x$, which we denote by 
$$g_j=(g_{j1},\cdots, g_{jn}):\mathbb{R}^n\rightarrow \mathbb{R}^n.$$
We set 
\begin{equation}\label{eqn: discreted surface element in $R^n$}
    \mu_j=(\mu_{j1},\mu_{j2},\ldots,\mu_{jn}):=N(y_j) \tau(y_j),
\end{equation}
then we have $\tau(y_j)=||\mu_j||$, and
\begin{equation}\label{discreted indicator function contain ing surface element in $R^n$}
    \chi(x)\approx \sum_{j=1}^{N_{Y}}g_j(x)\cdot\mu_j=\sum_{j=1}^{N_{Y}}\sum_{k=1}^{n}g_{jk}(x)\mu_{jk}.
\end{equation}
Let $X={\{ x_i\}}_{i=1}^{N_X=nN_{Y}}\subset \partial\Omega $ be another point set that also samples $\partial\Omega$, satisfying $X\cap Y=\emptyset$.
From equations \eqref{eqn: indicator function in $R^n$} and \eqref{discreted indicator function contain ing surface element in $R^n$}, we obtain 
\begin{equation}\label{the equation of indicator function}
    \sum_{j=1}^{N_{Y}}\sum_{k=1}^{n}g_{jk}(x_i)\mu_{jk}=1/2, \quad i = 1, 2, \ldots, N_X.
\end{equation}
By solving the above system of linear equations, we can obtain $\mu_{j}$ and hence get $\tau(y_j)=||\mu_j||$ as mentioned above.

A trouble in the above method is that $\frac{1}{|x-y_j|^n}$ can produce singularity and hence lead to big error if some $x\in X$ is close to $y_j$.
In \cite{LXSW22}, a method is proposed to soften the possible singularity.
We do not discuss the details of it here.

On the other hand, we are also interested in the case that the hyrpersurface $\partial\Omega$ is already known.
In this case of course we also know the interior $\Omega$ of $\partial\Omega$.
It follows that we can take $X$ to be a subset of $\Omega$ that keeps away from the boundary $\partial\Omega$ and then operate the above process to calculate the volume elements $\tau(y_j)$
by replacing the systems \eqref{the equation of indicator function} by the system
\begin{equation}
    \sum_{j=1}^{N_{Y}}\sum_{k=1}^{n}g_{jk}(x_i)\mu_{jk}=1, \quad i = 1, 2, \ldots, N_X.
\end{equation}
In this way, we never meet singularity since $Y$ lies in $\partial\Omega$.
Our modification of the method in \cite{LXSW22} is particularly effective if the equation of $\partial\Omega$ is already known.

\subsection{Hypersurfaces with Boundaries in $\mathbb{R}^n$}
In the previous subsection, we have discussed the integration of functions on compact hypersurfaces in $\mr^n$ without boundary.
That method can not be directly applied to hypersfuraces with boundary since for this case Lemma \ref{lem:Gauss formula} is never valid again.
In this section, we consider hypersurfaces with boundary and try to reduce them to the cases of hypersurfaces without bounday.
In this process, we need more information than samples of points of the hypersurfaces, namely, we also need know the normal vectors of the hypersurfaces at the sampling points.

We are know going to discuss the details.
Let $M$ be a compact hypersurface with boundary in $\mathbb{R}^n$, $Y = \{y_j\}_{j=1}^{N_Y}\subset M$ be a sample of points, 
and $N({y_j})$ be the outward-pointing unit normal vector of $M$ at $y_j$.
The insight is to consider a solid $\epsilon$-collar $\Omega$ of $M$, then the boundary $\partial\Omega$ gives us a compact (piecewise smooth) hypersurface without boundary (as shown in Figure 1).
The exact formula $\Omega$ is given by 
$$\Omega=\{y+tN(y)|y\in M, 0\leq t\leq \epsilon\},$$
and $\partial \Omega$ is given by 
$$\partial\Omega=M\cup\{y+\epsilon N(y)|y\in M\}\cup \{y+tN(y)|y\in \partial M, 0\leq t\leq \epsilon\}.$$

\begin{figure}[!h]
    \centering

        \includegraphics[width=0.6\textwidth]{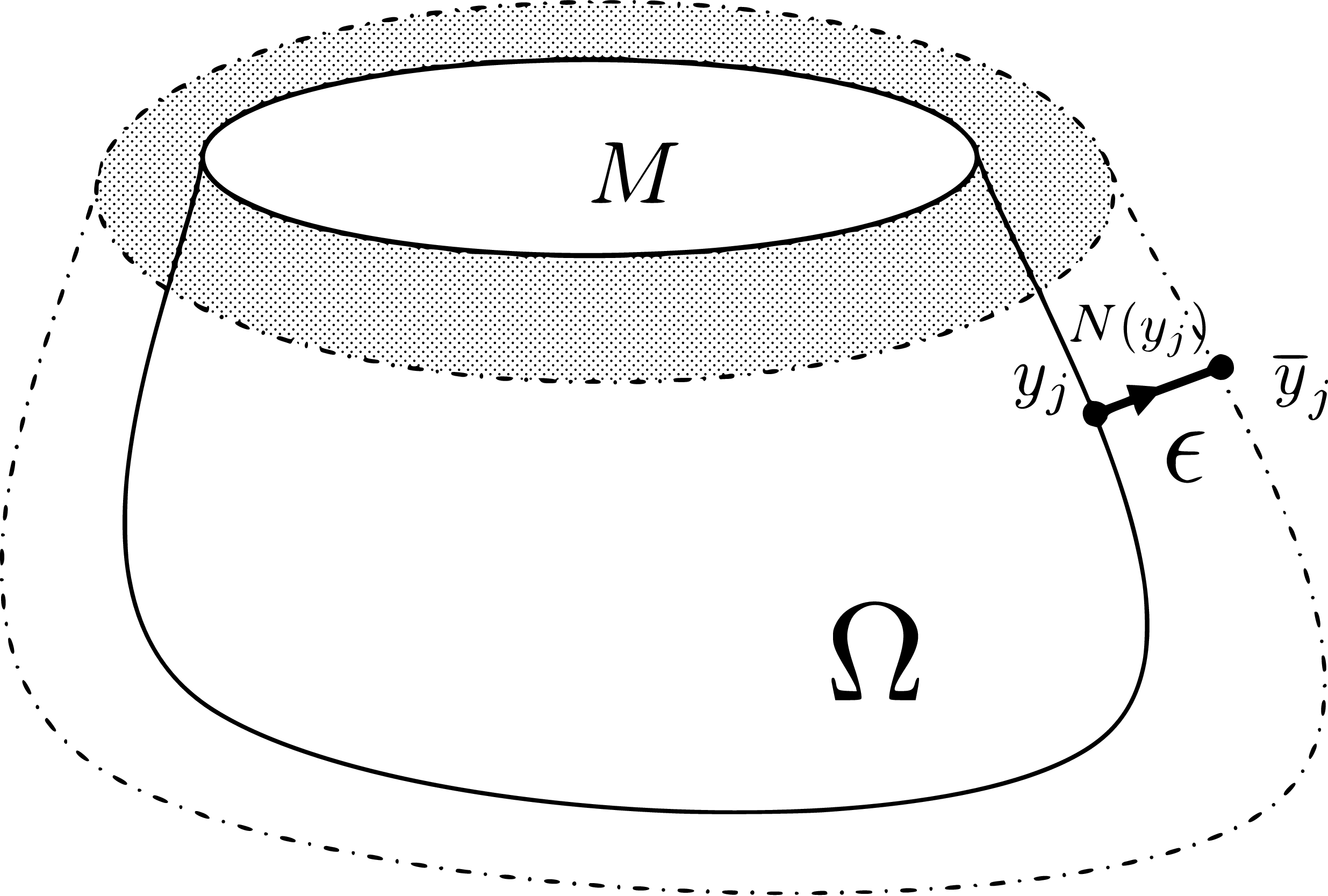}
    \caption{hypersurface with boundaries} 

\end{figure}

We now extend the sample $Y$ of points of $M$ to a sample $\tilde Y$ of points in $\partial \Omega$ by setting
$$\tilde Y=Y\cup\{\bar{y_j}:=y_j+\epsilon N(y_j)|y_j\in Y\}.$$
We then repeat the process in the previous section to calculate the volume elements $\{\tau(y_j), \tau(\bar{y_j})\}$. 
Indeed, in the present case, the number of the linear equations in  \eqref{the equation of indicator function} can be reduced to $2N_Y$
since $N(y_j)$ are already known and we can reasonably assume $N(\bar{y_j})=-N(y_j)$.

Let $f$ be a continuous function on $M$. 
Then the digital approximation of the integration $\int_Mf d\tau$ is given by 
$$\int_M fd\tau\approx \frac{1}{2}\sum^{N_Y}_{j=1}f(y_j)(\tau(y_j)+\tau(\bar{y_j})).$$
It is natural to expect that the right hand side approximates the real integral efficiently if $\epsilon$ is sufficiently small and $Y$ samples $M$ very well.

\subsection{Low-dimensional submanifold in $\mathbb{R}^n$}
The previous two subsections consider integration of functions on hypersurfaces in $\mr^n$, i.e., on submanifolds in $\mr^n$ of codimension 1.
In this subsection, we consider integration on general submanifolds in $\mr^n$ of higher codimension with or without boundary.
Let $M\subset\mr^n$ be a compact submanifold of codimension $r$ with boundary $\partial M$ ($\partial M=\emptyset$ if $M$ is boundary free).
Our trick is to consider tubular neighborhoods of $M$ (shown as in Figure 2) and reduce the calculation of integrations on $M$ to the calculation of integrations 
on the boundaries of the tubular neighborhoods, which are (piecewise smooth) compact hypersurface in $\mr^n$ without boundary,
then the method in \S \ref{subsec:hypersurface without boundary} applies.

\begin{figure}[!h]
    \centering
    \includegraphics[scale=0.3]{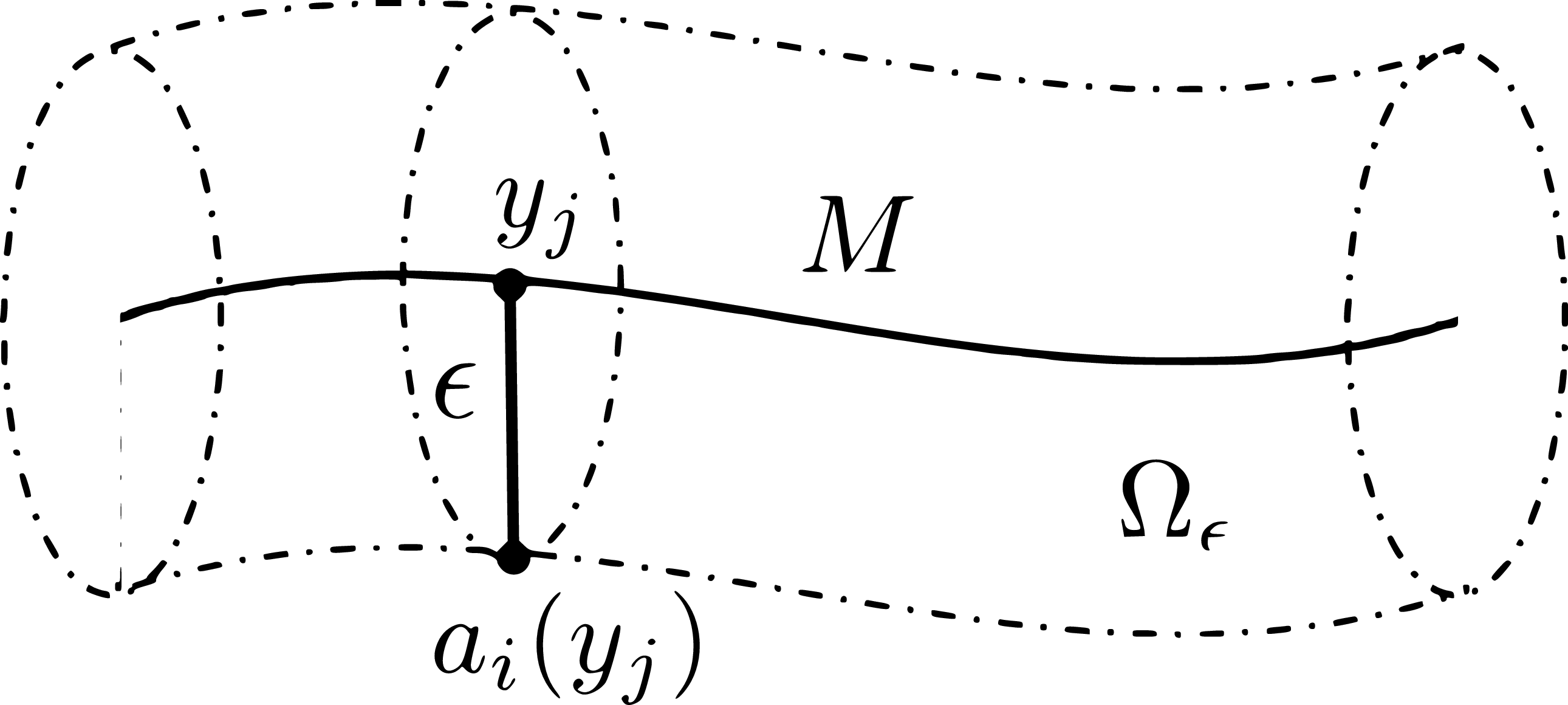}
    \caption{Low dimensional submanifold}
\end{figure}

If $N_1(y),\cdots, N_r(y)$ is the orthonormal basis of $Y$ at $y$, then the $\epsilon$-tubular neighborhood of $Y$ is given by 
$$\Omega_\epsilon=\{y + \sum_{k=1}^{r}\epsilon_{k}N_k(y)|y\in Y, \sum_{k=1}^{n-m}(\epsilon_{k})^2 =\leq \epsilon^2\}.$$
Its boundary is given by 
$$\partial\Omega_\epsilon=\left\{y + \sum_{k=1}^{r}\epsilon_{k}N_k(y)|y\in Y, \sum_{k=1}^{n-m}(\epsilon_{k})^2 = \epsilon^2\right\}\cup \left\{y + \sum_{k=1}^{r}\epsilon_{k}N_k(y)|y\in \partial Y, \sum_{k=1}^{n-m}(\epsilon_{k})^2 \leq \epsilon^2\right\},$$
which is a piece-wise smooth closed hypersurface in $\mr^n$.

The initial information that we need for the calculation consists of a sample of points $Y=\{y_j\}^p_{j=1}$ in $M$,
and a orthonormal basis $\{N_1(y_j),\cdots, N_p(y_j)\}$ of the $r$-dimensional normal subspace of $M$ at $y_j$ for $1\leq j\leq p$.
For the calculation of integrations of functions on $\partial\Omega_\epsilon$, we need to construct from the initial datum 
$(y_j; N_1(y_j),\cdots, N_r(y_j))^p_{j=1}$ a sample of points in $\partial\Omega_\epsilon$.
For this purpose, we choose a sample of points $\{a_1, \cdots, a_q\}$ in the sphere 
$$S^r(\epsilon)=\{(t_1,\cdots, t_r)\in \mr^r| t^2_1+\cdots+t^2_r=\epsilon\}$$
in $\mr^r$.
This gives us points say $\{a_1(y_j),\cdots, a_q(y_j)\}$ in the slice 
$$\left\{y_j + \sum_{k=1}^{r}\epsilon_{k}N_k(y_j)|\sum_{k=1}^{n-m}(\epsilon_{k})^2 = \epsilon^2\right\}$$
of $\partial\Omega_\epsilon$.
Then it is reasonable to set 
$$\tilde Y:=\{a_i(y_j)|1\leq i\leq q, 1\leq j\leq p\}$$
as a sample of points in $\partial\Omega_\epsilon$.
Now we can apply the process in \S \ref{subsec:hypersurface without boundary} to calculate the volume elements $\{\tau{a_i(y_j)}\}$.
Indeed, if we write $a_i=(a_{i1},\cdots, a_{ir})$, then the unit outward pointing normal of $\partial\Omega_\epsilon$ at $a_i(y_j)$ is given by 
$$y_j+a_{i1}N_1(y_j)+\cdots+a_{ir}N_r(y_j).$$
With these normal vectors, we can reduce the number of equations in the systems in  \eqref{the equation of indicator function} to $pq$.

Let $f$ be a continuous function on $M$.
Via the canonical projection $\pi:\Omega_\epsilon\ra M$, we get a function $\tilde f=f\circ\pi$ on $\Omega_\epsilon$.
As shown in \S \ref{subsec:hypersurface without boundary}, the integration $\int_{\partial\Omega_\epsilon}\tilde f$ can be approximated as:
$$\int_{\partial\Omega_\epsilon}\tilde f\approx \sum_{i, j}\tilde f(a_i(y_j))\tau(a_i(y_j))=\sum_{i, j}f(y_j)\tau(a_i(y_j)).$$
From Fubini's theorem, we can see that
$$\int_M fd\tau\approx \frac{1}{s_r}\sum_{i, j}f(y_j)\tau(a_i(y_j)),$$
where $s_r$ is the volume of the sphere $S^r(\epsilon)$.
We are done for digitally calculating integration of functions on $M$.

\section{Numerical calculation method for function integration in Riemannian manifolds}

In the previous section, we have considered the integral of a function on compact hypersurfaces with or without boundary, and low-dimensional submanifolds in $R^n$. 
In this section, we will prove a key theorem similar to lemma \ref{lem:Gauss formula}, which will play a crucial role in extending the above method to the case of Riemannian manifolds. 
We will focus on hypersurfaces in compact Riemannian manifolds since fundamental solutions on compact Riemannian manifolds satisfy certain good estimates.
Similar method can be extended to noncompact manifolds whose fundamental solutions satisfies similar estimates.

Firstly, we fix some notations. let $M$ be a compact oriented Riemannian manifold of real dimension dim$M$ = n, with smooth Riemannian metric $g$ and $\Omega$ is a relatively compact domain with smooth boundary in the Riemannian manifold. Let $N$ be the outward-pointing unit normal vector field on $\partial \Omega$. Now we start to prove the theorem, which expresses the 0-1 indicator function of $\Omega$ in terms of certain integral on the boundary.

\begin{thm}
    Let $\Omega \subset (M,g)$ be a bounded open set with smooth boundary $\partial \Omega$. Then
   
   \begin{equation} \label{eqn: indicator function}
    \chi(p):= - \int_{\partial \Omega} g ( \nabla_q G(p, q), N(q)  )d\tau(q) =
    \begin{cases}
        0 , p= M\setminus \overline{\Omega}\\
        \frac{1}{2},p=\partial \Omega\\ 
        1,p\in \Omega.\\
    \end{cases}
   \end{equation}
   where $d\tau(q)$ is the area form of $\partial \Omega$, and $G(p, q)$ is the fundamental solution to the Laplace equation, i.e.
   \begin{equation}
        \Delta_{q}^{distr.}G(p,q) =\delta_p
   \end{equation} 
   in the sense of distributions, where $\delta_p$ is the Dirac measure at $p$.
\end{thm}

\begin{proof}
    First, we prove  $\chi(p)$ is defined for all $p\in M$. For any $p \in M\setminus\partial\Omega$, the indicator function $\chi(p)$ is well defined  by the well definition of the double layer potential, see (\cite{MT0})(\cite{MT}). 
    
   Consider $p \in \partial \Omega$. It suffices to show that $$\chi(p)= - \int_{\partial \Omega} g ( \nabla_q G(p, q), N(q)  )d\tau(q)$$ is uniformly bounded independent of $p$. We suppose $(U, \varphi , x^i)$ are Riemannian normal coordinates charts and take a geodesic ball with a small appropriately radius $\frac{r_0}{4} < \epsilon < r_0$ as the exponential coordinates neighborhood of point p such that $$B(p, \epsilon) = exp_{p}(B_{p}(\epsilon)), \varphi(p) = x := (x_1, x_2,\ldots, x_n) = 0. $$ 
   we set $$\varphi(q) = y := (y_1, y_2,\ldots, y_n) $$ for each $q \in B(p, \epsilon) \cap \partial \Omega$.

    We suppose that $r_0$ is injective radius of $(M, g)$ and  $f(r)$ is a positive decreasing function such that:
    \begin{equation}
        f(r) = 
        \begin{cases}
            1, r < \frac{r_0}{2}\\
            0, r > r_0\\
        \end{cases}, \quad  0 \leq f(r) \leq 1 \quad and \quad |f'(r)| < \frac{100}{r_0},
    \end{equation}
    where $r = d(p,q)$.
    Set $$H(p.q) = \frac{f(r)}{(n-2)\omega_{n}r^{n-2}}$$ and we obtain $$|\bigtriangledown_{q}H(p,q)| = |\frac{f'(r)}{(n-2)\omega_{n}r^{n-2}} - \frac{f(r)}{\omega_{n}r^{n-1}}| \leq cr^{1-n}$$ where $c$ is a constant.
    
    From the proof of Theorem4.17 in Aubin (\cite{A82}), We also define $\varGamma(p,q) = \varGamma_1(p,q) = - \Delta_{q} H(p,q)$, $\varGamma_{i+1}(p,q)=\int_{M}\varGamma_i(p,t)\varGamma(t,q)d\tau(t)$ and $H_{i+1}(p,q) =\int_{M}\varGamma_i(p,t)H(t,q)d\tau(t)$ for $i \in \mathbb{N} $. Let $ k = [\frac{n}{2}]+1$, where $[\frac{n}{2}]$ is the greatest integer not exceeding $\frac{n}{2}$. Then we obtain the expression of the fundamental solution:
    \begin{equation}
        G(p,q) = H(p,q) + \sum_{i=2}^{k}H_i(p,q) + F(p,q)
    \end{equation}
    where $\Delta_{q} F(p,q) = \varGamma_{k+1}(p,q) $, such that $$\Delta_{q}^{distr.}G(p,q) = \delta_p.$$
    
    We need the following estimate of the gradient of the the fundamental solution $G(p,q)$.
    By lemma \ref{lem: the estimate of integration} ,we get the estimates as follows:
    \begin{equation}
        |\varGamma_i(p,q)| \leq 
        \begin{cases}
          \frac{C_0}{d(p,q)^{n-2i}}, \quad if \quad 2i < n,\\
          C_0(1 + |\ln d(p,q)|), \quad if \quad 2i =n,\\
          C_0, \quad if \quad 2i >n.\\
        \end{cases} 
    \end{equation}
    and 
   \begin{equation}
    |\bigtriangledown_q H_{i+1}(p,q)| \leq \int_{M}|\varGamma_{i}(p,t)|\cdot|\bigtriangledown_qH(t,q)d|\tau(t) \leq
       \begin{cases}
          \frac{C_1}{d(p,q)^{n-2i-1}}, \quad if \quad 2i+1 < n,\\
          C_1(1 + |\ln d(p,q)|), \quad if \quad 2i+1 =n,\\
          C_1, \quad if \quad 2i+1 >n.\\
       \end{cases}
   \end{equation}
   where $C_0, C_1$ are constants independent of $p$.
   Therefore, from the appendix A of \cite{DHR04}, there exsits a constant $C > 0$ depending only on $(M, g)$ such that $ ||F(p, \cdot)||_{C^{2, \theta }} < C$. So $$ |\bigtriangledown_q F(p,q)| < C.$$


   In $\partial \Omega \cap B(p, \frac{r_0}{4})$, utilizing \ref{eqn: exponential coordinates} and lemma \ref{lem: estimate of normal vector}, we can get 
   $$ \int_{\partial \Omega\cap B(p,\frac{r_0}{4})}|\frac{\partial H(p,q)}{\partial N(q)}| d\tau(q) = \int_{\phi(\partial \Omega) \cap B_p(\frac{r_0}{4})} |\frac{(y-x)\cdot N(y)}{\omega_n|x-y|^{n}}|d\tau(y) \leq \int_{\phi(\partial \Omega) \cap B_p(\frac{r_0}{4})} \frac{1}{\omega_n|x-y|^{n-2}} d\tau(y) \leq C_2 $$ where $C_2$ is constants independent of $p$.

   In $\partial \Omega \cap B(p, \frac{r_0}{4})^{c}$, because $\partial \Omega$ is compact, we have $$\int_{\partial \Omega \cap B(p,\frac{r_0}{4})^{c}} |\frac{\partial G(p,q)}{\partial N(q)}|d\tau(q) \leq C_3$$ where $C_3$ is independent of $p$.

   Therefore, we have:
   \begin{equation}
    \begin{aligned}
      |-\int_{\partial \Omega}\frac{\partial G(p,q)}{\partial N(q)}d\tau(q)|
       & \leq \int_{\partial \Omega}|\frac{\partial G(p,q)}{\partial N(q)}| d\tau(q)\\
      & \leq \int_{\partial \Omega\cap B(p,\frac{r_0}{4})} |\frac{\partial G(p,q)}{\partial N(q)}|d\tau(q) + \int_{\partial \Omega\cap B(p,\frac{r_0}{4})^{c}} |\frac{\partial G(p,q)}{\partial N(q)}|d\tau(q)\\
      & \leq \int_{\partial \Omega\cap B(p,\frac{r_0}{4})}|\frac{\partial H(p,q)}{\partial N(q)}| + \sum_{i=2}^{k}|\frac{\partial H_i(p,q)}{\partial N(q)}| + |\frac{\partial F(p,q)}{\partial N(q)}|d\tau(q) +C_3\\
      & \leq C_4.
    \end{aligned}
\end{equation}
    where $C_4$ is only depends on $(M, g)$. So \eqref{eqn: indicator function} is defined for $p\in M$.

  Next, we prove that the right hand side of (\eqref{eqn: indicator function}) is vaild, then we consider three cases. For the first case, we assume that $p \in M\setminus \partial\Omega$. By the theorem \ref{thm: the divergence theorem}, we know $$ 0 = \int_{\Omega}\Delta_qG(p, q) d\sigma(q) = \int_{\partial \Omega} g(\nabla G(p, q), N(q)) d\tau(q),$$ where $d\sigma(q)$ is the volume form of $\partial \Omega$ at $q$.

  For the second case, we assume that $p \in \Omega$. Let $B(p, \epsilon)$ is a geodesic ball on $\Omega$. Indeed, utilizing the theorem \ref{thm: the divergence theorem}, we have $$ 0 = \int_{\Omega - B(p, \epsilon )} \Delta_q G(p, q) d\sigma(q) = \int_{\partial \Omega} g(\nabla_q G(p, q), N(q))d\tau(q) + \int_{\partial B(p, \epsilon)} g(\nabla_q G(p, q), N(q))d\tau(q),$$ which is equivalent to the following: $$ \chi(p) = \int_{\partial B(p, \epsilon)} g(\nabla_q G(p, q), N(q))d\tau(q).$$ Now suppose $(U, \varphi , x^i)$ are Riemannian normal coordinates charts. Therefore, we can obtain
    \begin{equation}
        \begin{aligned}
            \int_{\partial B(p, \epsilon)} &g(\nabla_q G(p, q), N(q))d\tau(q)\\
            &=  \int_{\partial B_p(\epsilon)} |\frac{(y-x)\cdot N(y)}{\omega_n|x-y|^{n}}|d\tau(y)  d\tau(y)+ O(\epsilon)\\
            &= \int_{\partial B_p(\epsilon)} |\frac{(y-x)\cdot \frac{x-y}{|x-y|}}{\omega_n|x-y|^{n}}|d\tau(y)  d\tau(y)+ O(\epsilon)\\
            &= \int_{\partial B_p(\epsilon)} \frac{1}{\omega_n|x-y|^{n-1}}d\tau(y)  d\tau(y)+ O(\epsilon)\\
            &= 1 + O(\epsilon).
        \end{aligned} 
    \end{equation}
 So we have $$\chi(p) = 1 + O(\epsilon).$$
      
For the last case, we assume that $p \in \partial \Omega$. Set 
   $$\Omega_{\epsilon} = \Omega - (\Omega \cap B(p, \epsilon)), C_{\epsilon} = \{q \in \partial B(p, \epsilon),N(p)\cdot q < 0 \}, C_{\epsilon}^{'} = \partial M_{\epsilon}\cap C_{\epsilon}.$$ Similarly, by the theorem \ref{thm: the divergence theorem}, we have
   \begin{equation}
    \begin{aligned}
        0 &= \int_{\Omega_{\epsilon}}\Delta_qG(p, q)d\sigma(q) \\
        &= \int_{\partial \Omega_{\epsilon}} g(\nabla G(p, q), N(q)) d\tau(q)\\
        &= \int_{\partial \Omega_{\epsilon} - C_{\epsilon}^{'}} g(\nabla G(p, q), N(q)) d\tau(q) + \int_{C_{\epsilon}^{'}} g(\nabla G(p, q), N(q)) d\tau(q).
    \end{aligned}
   \end{equation}
Observe that 
\begin{equation}
    \begin{aligned}
        \int_{C_{\epsilon}^{'}} g(\nabla G(p, q), N(q)) d\tau(q) &= \int_{\varphi(C_{\epsilon}^{'})}  \frac{1}{\omega_n|x-y|^{n-1}}d\tau(y) + O(\epsilon)\\
        &= \frac{1}{\omega_{n}\epsilon^{n-1}}\int_{C_{\epsilon}^{'}} d\tau(q)+ O(\epsilon)
    \end{aligned}
\end{equation} 
We want to show that $\int_{C_{\epsilon}^{'}} d\tau(q) = \int_{C_{\epsilon}} d\tau(q) + O(\epsilon^n).$ The surface area is clearly determined by the product of the base area and the height. By the Taylor's Theorem, for every $q \in \partial B(p, \epsilon)$, we have $|(p-q)\cdot N(p)| < O(\epsilon^2)$. Additionally, the base area is $O(\epsilon^{n-2})$. Thus $\int_{C_{\epsilon}^{'}} d\tau(q) = \int_{C_{\epsilon}} d\tau(q) + O(\epsilon^n) = \frac{1}{2}\omega_{n}\epsilon^{n-1} + O(\epsilon^n)$.
Combining the above equations, we infer that
$$\int_{\partial \Omega_{\epsilon} - C_{\epsilon}^{'}} g(\nabla G(p, q), N(q)) d\tau(q) = -\frac{1}{2} +O(\epsilon).$$ Taking the limit as $\epsilon \rightarrow 0$, we prove $$\int_{\partial \Omega} g(\nabla G(p, q), N(q)) d\tau(q) = -\frac{1}{2}.$$

 \end{proof}

 Summarizing the method in \S \ref{subsec:hypersurface without boundary}, the algorithm to calculate the integral of a function on the boundaries related compact domains in Riemannian manifolds involves four steps as follows:
 
 $\mathbf{Step 1.}$
 Construct the indicator function. For a domain $\Omega$ with smooth boundary $\partial \Omega$ on $(M,g)$, we have $$\chi(p)= - \int_{\partial \Omega} g ( \nabla_q G(p, q), N(q)  )d\tau(q).$$
  
 $\mathbf{Step 2.}$ 
 Discretize the indicator function. Assume that a point set $\mathcal{Q}={\{ q_j\}}_{j=1}^{N_{\mathcal{Q}}}\subset \partial\Omega$ is given. In local coordinates, let $$\nabla G(p-q_j) = G^{jk}(q)\frac{\partial}{\partial x^k}|_{q_j}, $$ and $$N(q_j)\cdot \tau(q_j) = \mu^{jl}(q_j)\frac{\partial}{\partial x^l}|_{q_j}$$ where $k,l = 1,2,\ldots, n$. Then by combining the above equations, we get the discrete version of the indicator function: 
 \begin{equation}
    \begin{aligned}
        \chi(p) &= \int_{\partial \Omega} g (\nabla G(p,q)\cdot N(q))d\tau(q)\\
        &\approx \sum_{j=1}^{N_{\mathcal{Q}}}g(\nabla G(p, q_j), N(q_j))\tau(q_j)\\
        &= \sum_{j=1}^{N_{\mathcal{Q}}}g(\nabla G(p, q_j), N(q_j)\cdot \tau(q_j))\\
        &= \sum_{j=1}^{N_{\mathcal{Q}}}\sum_{k,l=1}^{n}G^{jk}(p)\mu^{jl}(q_j)g_{kl}(q_j)
    \end{aligned}
 \end{equation}

 $\mathbf{Step 3.}$
 Calculate the volume elements $\tau (q_j)$. 
 Let $\mathcal{P}={\{ P_i\}}_{i=1}^{N_{\mathcal{P}}} \subset \Omega$ be another point set that also samples $\partial \Omega$, satisfying $\mathcal{P} \cap \mathcal{Q} = \emptyset$. We can obtain the following equations: $$\chi(p_i) = \sum_{j=1}^{N_{\mathcal{Q}}}\sum_{k,l=1}^{n}G^{jk}(p_i)\mu^{jl}(q_j)g_{kl}(q_j) = 1/2, \quad i = 1,2,\ldots, N_{\mathcal{P}}.$$ By solving the above system of linear equations, we can obtain $\mu^{jl}(q_j)$ and hence get $\tau (q_j) = ||\mu^{jl}(q_j)||$.

 $\mathbf{Step 4.}$ 
 Calculate the integral $\int_{\partial \Omega} f d\tau$.  For any $ \in C^{0}(M)$, we have $$\int_{\partial\Omega}f(p)d\tau(p) = \sum_{j=1}^{N_{\mathcal{Q}}}f(q_j)\tau_j = \sum_{j=1}^{N_{\mathcal{Q}}}f(q_j)\tau(q_j). $$

The above method for digital calculation of function integration 
is efficient in the case that the explicit form of the fundamental solution on $M$ is known.
Such manifolds often has large symmetry, for example, Riemannian symmetric spaces.

\bibliographystyle{amsplain}

\end{document}